\documentclass[11pt]{amsart}

\usepackage{etex}
\usepackage{amsmath, amssymb}
\usepackage[frame,cmtip,arrow,matrix,line,graph,curve]{xy}
\usepackage{graphpap, color, paralist, pstricks}
\usepackage[mathscr]{eucal}
\usepackage[pdftex]{graphicx}
\usepackage[pdftex,colorlinks,backref=page,citecolor=blue]{hyperref}
\usepackage{tikz}

\newtheorem{theorem}{Theorem}[section]
\newtheorem{proposition}[theorem]{Proposition}
\newtheorem{corollary}[theorem]{Corollary}
\newtheorem{lemma}[theorem]{Lemma}

\newtheorem{problem}[theorem]{Problem}

\theoremstyle{definition}
\newtheorem{definition}[theorem]{Definition}
\newtheorem{example}[theorem]{Example}

\newtheorem{question}[theorem]{Question}
\newtheorem{remark}[theorem]{Remark}

\newcommand{\PP}{\mathbb{P}}
\newcommand{\QQ}{\mathbb{Q}}
\newcommand{\CC}{\mathbb{C}}

\newcommand{\cO}{\mathcal{O} }

\newcommand{\cC}{\mathcal{C} }
\newcommand{\cE}{\mathcal{E} }
\newcommand{\cF}{\mathcal{F} }

\newcommand{\cH}{\mathcal{H} }

\newcommand{\cK}{\mathcal{K} }

\newcommand{\cQ}{\mathcal{Q} }

\newcommand{\cZ}{\mathcal{Z} }

\newcommand{\rE}{\mathrm{E} }
\newcommand{\rH}{\mathrm{H} }
\newcommand{\rM}{\mathrm{M} }

\newcommand{\rP}{\mathrm{P} }

\newcommand{\proj}{\mathrm{Proj}\;}

\newcommand{\cHom}{\mathcal{H}om}

\def\Hom{\mathrm{Hom} }
\def\Ext{\mathrm{Ext} }
\def\cExt{\mathcal{E}xt }

\def\lr{\rightarrow}

\def\bH{\mathbf{H}}
\def\bM{\mathbf{M}}
\def\bN{\mathbf{N}}

\newcommand{\ses}[3]{0\lr{#1}\lr{#2}\lr{#3}\lr 0}

\hypersetup{%
pdfauthor={Kiryong Chung},%
pdfkeywords={.},%
citecolor=blue,%
linkcolor=blue,%
}

\begin{document}

\title[Compactified spaces of rational quartic plane curves]{Sheaf theoretic compactifications of the space of rational quartic plane curves}

\author{Kiryong Chung}
\address{Department of Mathematics Education, Kyungpook National University, 80 Daehakro, Bukgu, Daegu 41566, Korea}
\email{krchung@knu.ac.kr}

\keywords{Plane rational curves, Compactification, Birational geometry}
\subjclass[2010]{14E15, 14E05, 14M15}

\begin{abstract}
Let $R_4$ be the space of rational plane curves of degree $4$. In this paper, we obtain a sheaf theoretic compactification of $R_4$ via the space of $\alpha$-semistable pairs on $\PP^2$ and its birational relations through wall-crossings of semistable pairs. We obtain the Poincar\'e polynomial of the compactified space.
\end{abstract}

\maketitle

\section{Introduction}
\subsection{Compactifications and results}
Let $R_d$ be the space of \emph{irreducible} rational plane curves of degree $d$. The space $R_d$ is an irreducible variety of dimension $3d-1$ (\cite{KP01}). One of very famous compactification of $R_d$ is so called the moduli space of stable maps, which is a key role in the curve counting theory. For the detailed description, see \cite{FP97}. This compactification is motivated by viewing the rational plane curve with its unique normalization map (up to isomorphism). In this paper, our aim is to obtain a sheaf-theoretic compactification of $R_d$ and its birational relations. In this case, the curve is regarded as the pair of the direct image sheaf of the normalizaition map and its canonical section. In detail,
\begin{enumerate}
\item let $\rM(dm+1)$ be the moduli space of stable sheaves $F$ with Hilbert polynomial $\chi(F(m))=dm+1$ and 
\item $\rM^{\alpha}(dm+1)$ be the moduli space of $\alpha$-semistable pairs $(s,F)$ with Hilbert polynomial $\chi(F(m))=dm+1$.
\end{enumerate}
For explicit definition and properties, see Section \ref{revelent}. Note that all of these moduli spaces are projective scheme (\cite{Sim94, He98, LP93b}). For the irreducible plane curve $C$ of degree $d$, let $n:C^{v}\lr C$ be the normalization map such that $C^{v}\cong \PP^1$. Then the pair (resp. sheaf $n_{*}\cO_{C^{v}}$) $(s, n_{*}\cO_{C^{v}})$ with the unique section $s\in \rH^0(n_{*}\cO_{C^{v}})$ is $\alpha$-stable (resp. stable) for any positive parameter $\alpha>0$ (\cite[Proposition 3.18]{CC11}). That is, we have natural injective maps
\[
R_d\subset\rM(dm+1);\;[C]\mapsto n_{*}\cO_{C^{v}}\text{ and }
R_d\subset\rM^{\alpha}(dm+1);\;[C]\mapsto (s, n_{*}\cO_{C^{v}}).
\]
Let us denote $\bM_d$ (resp. $\bM_d^{\alpha}$) by the closure of the rational curves space $R_d$ in the moduli space $\rM(dm+1)$ (resp. $\rM^{\alpha}(dm+1)$).
\begin{problem}
Study the geometry $\bM_d$ via its variations spaces $\bM_d^{\alpha}$.
\end{problem}
\begin{remark}
For the degree is $d\leq 3$, one can easily see that all of $\alpha$-stable sheaves spaces $\bM_d^{\alpha}$ are isomorphic to $\bM_d$.
For $d=1$ and $2$, the space $\bM_d$ is isomorphic to the complete linear system $|\cO_{\PP^2}(d)|$.
The space $\bM_3$ is isomorphic to a $\PP^6$-bundle over $\PP^2$. Here the fiber $\PP^6$ parameterizes the rational cubic curves passing through a singular point $\PP^2$.
\end{remark}
For the higher degree cases since there does not have any systematic method to study the geometry of $\bM_d$, we do the work for case $d=4$ firstly. We will prove that there exists a divisorial contraction among $\bM^{\alpha}_d$. Let us state main theorem of this paper. For sufficiently large parameter $\alpha>>1$, let denote by $\bM_{d}^{\alpha \!>\!>1}:=\bM_{d}^{+}$.
It is known that the space $\bM_d^{+}$ is isomorphic to the relative Hilbert scheme of $\frac{(d-1)(d-2)}{2}$-points over the complete linear system $|\cO_{\PP^2}(d)|$ (\cite[Section 4.4]{He98}). We call that a birational contraction is a \emph{smooth} blow-down whenever the two birational varieties are locally smooth around the exceptional divisor and its image in the target space.
\begin{theorem}[\protect{Proposition \ref{mainprop}}]\label{thm:mainpropintro}
Under the above notation, there exist a smooth blow-down
\[p_{R}:\bM_4^{+}\longrightarrow \bM_4\]
where the exceptional divisor parameterizes the pair of the (degenerated) rational cubic curve and the colinear three points.
\end{theorem}
The key ingredient of the the proof of Theorem \ref{thm:mainpropintro} is the analysis of the restrictions of pair wall-crossing of $\rM^{\alpha}(4m+1)$. Using this fact, one can compute the virtual Poincar\'e polynomial of the space $\bM_4$ (Corollary \ref{maincor}).
\begin{corollary}
The virtual Poincar\'e polynomial of $\bM_4$ is given by
\[
1+2t^2+5t^4+8t^6+11t^8+12t^{10}+13t^{12}+13t^{14}+11t^{16}+7t^{18}+3t^{20}+t^{22}.
\]
\end{corollary}
\begin{question}
Find out a smooth resolution of the space $\bM_4$ and study its geometry in the view point of curve counting theory.
\end{question}
In the introduction part of the paper \cite{DH88}, S. Diaz and J. Harris has already mentioned the necessity of a \emph{good} compactification of the space $R_d$ for the intersection theory.
\subsection{Stream of the paper}
In section \ref{revelent}, we explicitly propose relevant spaces and review its properties. In section \ref{profthm}, we prove our main theorem by analyzing the wall-crossing of $\alpha$-stable pairs on $\PP^2$. In the last section, as presenting the special fibres by using \texttt{Macaulay2} (\cite{GS}), we calculate the Poincar\'e polynomial of $\bM_4$.

\textbf{Acknowledgement.}
The author would like to thank Dawei Chen for suggesting this topic and Jinhyung Park for helpful discussions and comments.
\section{Relevant moduli spaces}\label{revelent}
In this section, we collect known facts for the proof of Theorem \ref{thm:mainpropintro}
\subsection{Kronecker quiver and Hilbert scheme of cubes}\label{sub:defq}
Let $\bN:=\bN(3;2,3)$ be the moduli space of the space of sheaf homomorphisms (so called, the Kronecker quivers)
\begin{equation}\label{res1}
\cO_{\PP^2}(-2)^{\oplus 2}\longrightarrow \cO_{\PP^2}(-1)^{\oplus 3}
\end{equation}
up to the action of the automorphism group $\mathrm{GL}_{2}\times \mathrm{GL}_{3}/\CC^*$. It is known that $\bN$ is a smooth projective variety of dimension $6$ (\cite{Dre87}). On the other hand, let $Z$ be a finite subscheme of $\PP^{2}$ of length $3$. Since a resolution of a general ideal sheaf $I_Z(1)$ twisted by $\cO_{\PP^{2}}(1)$ is of the form \eqref{res1}, 
the moduli space $\bN$ is birational to the Hilbert scheme $\bH[3]$ of three points in $\PP^2$. 
The classification of locus in $\bH[3]$ are very well-knonwn. We list some loci which we will use later. Let $x,y,z$ be the homogeneous coordinates of $\PP^2$.
\begin{enumerate}[a)]
\item Let $D_1$ be the locus of \emph{colinear} three points on $\PP^2$ which is isomorphic to a $\PP^3$-bundle over $\PP^2$.
\item Let $D_2$ be the locus of \emph{curvelinear} three points which is isomorphic to $\CC^*\times \CC$-bundle over $\PP^2$. The general element is of the form $\langle x^2, xy, y^2-zx\rangle$.
\item Let $D_3:=\overline{D}_2 \setminus D_2$ be the locus of the \emph{non-curvelinear} three points which is isomorphic to $\PP^2$. The general element is of the form $\langle x^2,xy,y^2\rangle$.
\end{enumerate}
An explicit birational relation between $\bH[3]$ and $\bN$ has been studied in \cite{LQZ03}.
\begin{proposition}[\protect{\cite[Section 6]{LQZ03}}]\label{prop:hilbcon}
Under the above notation, there exists a smooth blow-down morphism
\[
	t: \bH[3]\longrightarrow \bN
\]
which contracts the fiber $\PP^3$ of the exceptional divisor $D_1$.
\end{proposition}
\subsection{The moduli of stable maps space}
Let $C$ be a projective connected reduced curve. A map $f: C \to \PP^r$ is called by \emph{stable} if $C$ has at worst nodal singularities and $|\mathrm{Aut}(f)|<\infty$. Let $\cK_d$ be the moduli space of isomorphism classes of stable maps $f:C\to X$ with genus $g(C)=0$ and $\mathrm{deg} (f)=d$. The space $\cK_d$ has been widely studied in \cite{FP97}. In special, the space is compact and naturally contains the space $R_d$ as a dense open subset because the normalization map of an irreducible rational curve has the trivial automorphism.
\subsection{Summary of the result in \cite{CC17}}
Let $X$ be a projective variety with a fixed ample line bundle $\cO_X(1)$. Let \[\chi(F(m))=\sum_{k=0}^{\dim X} \dim \rH^k(F(m))\] be the Hilbert polynomial of the coherent sheaf $F\in \mathrm{Coh}(X)$.
\begin{definition}
A pair $(s,F)$, $s\in \rH(F)$ is called \emph{$\alpha$-semistable} if $F$ is pure and for any subsheaf $F'\subset  F$, the inequality
\begin{equation}\label{slop}
	\frac{\chi(F'(m))+\delta\cdot\alpha}{r(F')} \le
	\frac{\chi(F(m))+\alpha}{r(F)}
\end{equation}
holds for $m\gg 0$. Here $\delta=1$ if the section $s$ factors through $F'$ and $\delta=0$ otherwise. When the strict inequality holds, $(s,F)$ is called an \emph{$\alpha$-stable} pair. 
\end{definition}
It is well-known that there exists a projective scheme $\rM^{\alpha}( P(m))$ parameterizing $S$-equivalence classes of $\alpha$-semistable pairs with Hilbert polynomial $P(m)$ (\cite[Theorem 4.12]{LP93b} and \cite[Theorem 2.6]{He98}). By decreasing the stability parameter $\alpha$, one can obtain a sequence of flips among the moduli spaces of $\alpha$-stable pairs. 

As we let $\alpha=0$ in \eqref{slop}, one can obtain the usual concept of semistable sheaves and thus the moduli space $\rM(P(m))$ which parameterizes $S$-equivalence classes of semistable sheaves on $X$ with Hilbert polynomial $P(m)$ (For construction and its properties, see \cite{Sim94}). 
In this paper, we deal with the case $X=\PP^2$ and $P(m) = 4m + 1$.
In \cite{CC17,CM15}, the authors studied the geometric properties of $\rM^{\alpha}(4m+1)$ and $\rM(4m+1)$ using the wall-crossing of pairs. It turns out that there exists a unique wall at $\alpha=3$ (\cite[Lemma 3.1]{CC17}). Let us summarize the result for later use.
\begin{enumerate}[(i)]
\item (\cite[Section 4.4]{He98} and \cite{PT10}) Let $\rM^{+}( 4m+1):=\rM^{\alpha}( 4m+1)$for $\alpha>3$. The moduli space $\rM_4^{+}$ of $\alpha$-stable pairs is isomorphic to the relative Hilbert scheme $\cH[3]$ of three points on the complete linear system $|\cO_{\PP^2}(4)|$. For $(s,F)$ be a stable pair, the corresponding point in the relative Hilbert scheme $\cH[3]$ is given by
\[
\rM^{+}\lr \cH[3], (s,F)\mapsto (Z,C), Z\subset C=\mathrm{Supp}(F), l(Z)=3.
\]
\item (\cite[Proposition 4.4]{CC17}) Let $\rM^{-}( 4m+1):= \rM^{\alpha}( 4m+1)$ for $\alpha<3$. The forgetful map $r : \rM^{-} \to \rM: (s,F)\mapsto F$ is a smooth blow-up along the Brill-Noether locus $\{[F] \in \rM_4\;|\; \dim \rH^{0}(F) = 2\}$.
\item (\cite[Section 3.1]{CC17}) The single flip over $\rM^{3}(4m+1)$ is a composition of a smooth blow-up and a smooth blow-down. Let us denote $Y^{+}$  by the blow-up center in $\rM^{+}$. Then $Y^{+}$ is isomorphic to a $\PP^3$-bundle over $\PP^2\times \PP^9$ where the base parameterizes the pair of line and cubic curves and the fiber parameterizes three points on that line. 
\end{enumerate} 
\begin{remark}\label{remsection}
The image of the blow-up center $Y^{-}$ in $\rM^{-}$ by the forgetful map $r$ in (ii) parameterizes the stable sheaf $F$ which fits into a non-split extension:
\[\ses{\cO_C}{F}{\cO_L}\]
for a line $L$ and a cubic curve $C$. By a simple computation, we see that the image $r(Y^{-})$ is isomorphic to a $\PP(\Ext_{\PP^2}^1(\cO_L,\cO_C))\cong\PP^2$-bundle over $\PP^2\times \PP^9$. Also, 
\[ h^0(F) = \left\{ \begin{array}{ll}
         1 & \mbox{if $[F]\in \PP^2\setminus \PP^1$};\\
        2 & \mbox{if $[F]\in \PP^1$}.\end{array} \right. \]        
and thus $Y^{-}\cong r(Y^{-})$ by (ii).
\end{remark}
In summary, there exists a commutative diagram
\begin{equation}\label{diagram1}
\xymatrix{&\widetilde{\rM(4m+1)}\ar[dl]_{p^+}\ar[dr]^{p^-}&\\
\rM^{+}(4m+1)\ar@{<-->}[rr]^{\text{a single flip}}\ar[dr]&&\rM^{-}(4m+1)\ar[dl]\ar[d]^{r}\\
&\rM^{3}(4m+1)&\rM(4m+1).}
\end{equation}
Let $\bM_4^{\alpha}$ be the closure of the space $R_4$ of $\rM_4^{\alpha}(4m+1)$ for each $\alpha\in \QQ^{+}$. Then the diagram \eqref{diagram1} has been restricted to
\begin{equation}\label{diagram}
\xymatrix{&\widetilde{\bM}_4\ar[dl]_{p^+}^{\cong}\ar[dr]^{p^-}&\\
\bM_4^{+}\ar@{=>}[dd]^{\pi^+}\ar@{<-->}[rr]\ar[dr]&&\bM_4^{-}\ar[dl]\ar[d]_{r}^{\cong}\\
&\bM_4^{0}&\bM_4\ar@{=>}[d]^{\pi^-}\\
\bH[3]\ar[rr]^{t}&&\bN,}
\end{equation}
where $\pi^{+}: \bM_4^{+}\lr \bH[3]$ is the restriction map of the structure morphism of the relative Hilbert scheme $\cH[3]\lr \bH[3]$. The remaining one of this paper is to prove that
\begin{itemize}
\item the restrictions of maps $p^{+}$ and $r$ are an isomorphism,
\item the composition of the restriction maps
\[
p_R:=r\circ p^{-}\circ (p^{+})^{-1} :\bM_4^{+}\lr \bM_4
\]
is a smooth blow-up map and 
\item the exists a generically bundle morphism $\pi^{-}:\bM_4 \lr \bN$ such that $\pi^{-}\circ p_R=t\circ \pi^{+}$.
\end{itemize}
\section{Proof of Theorem \ref{thm:mainpropintro}}\label{profthm}
In this section, we analyze the restricted diagram in \eqref{diagram1} during the wall-crossing. One of main method is to use the elementary modification of sheaves coming from the family of stable maps.
\subsection{Some lemmas via stable maps}
\begin{lemma}[\protect{\cite[Example 6.1]{CC11}}]\label{stable}
For the irreducible plane curve $C$, let $f:C^v\lr C$ be the normalization map and $s\in\Hom(\cO_C, f_*\cO_{C^v})\cong \rH^0(\cO_{C^v})=\CC$ be the canonical section. Then the pair $(s,f_*\cO_{C^v})$ is $\alpha$-stable for all $\alpha\in \QQ^{+}$. In special, $f_*\cO_{C^v}$ is stable.
\end{lemma}
The correspondence between $R_4$ and the relative Hilbert scheme $\cH[3]$ is given by the following way. For the curve $C$, the direct image sheaf $f_*\cO_{C^v}$ fits into the short exact sequence
\[\ses{\cO_C}{f_*\cO_{C^v}}{\cQ}\]
such that $\dim \cQ=0$. By taking the dual $\cHom^{\bullet} (-, \cO_C)$ in the exact sequence,
\begin{equation}\label{short}
0\lr \cHom_C (f_*\cO_{C^v},\cO_C)\lr \cHom_C (\cO_C, \cO_C)=\cO_C \lr \cExt_C^1(\cQ, \cO_C)\lr 0
\end{equation}
we have a pair $(Z,C)$, $Z\subset C$ such that the subscheme $Z$ is defined by the (\emph{conductor}) ideal sheaf $I_{Z,C}=\cHom_C (f_*\cO_{C^v},\cO_C)$ (cf. \cite[Proposition B.5]{PT10}).
\begin{example}\label{conexam}
Let $f:\PP^1\lr \PP^2; [s:t]\mapsto [g_3(s,t)s:g_3(s,t)t:g_4(s,t)]$ be a map such that $g_i(s,t)$ are homogenous polynomial with respect to $s,t$ of degree $i$.
Then it is straightforward to check that the conductor ideal sheaf is defined by\[I_{Z}=\langle x^2,xy,y^2\rangle\]where $x,y,z$ are homogenous coordinates of $\PP^2$.
\end{example}
Recall that $\bM_4$ (resp. $\bM_4^{\alpha}$) is the closure of $R_4$ in $\rM(4m+1)$ (resp. $\rM^{\alpha}(4m+1)$) respectively.
\begin{proposition}\label{sectionspace}
For each closed point $[F]\in \bM_4$, 
\[\mathrm{dim}\rH^0(F)=1.\] 
Hence the forgetful map  $r$ is an isomorphism $\bM_4^{-}\stackrel{r}{\cong} \bM_4$.
\end{proposition}
\begin{proof}
Let $\cK_{4}$ be the moduli space of stable maps of degree $4$ into $\PP^2$.
Let
\[
\xymatrix{\cC\ar[r]^{\mathbf{f}}\ar[d]^{\pi}&\PP^2\\
\cK_4&}
\]
be a flat family of stable maps parameterized by $\cK_4$.
Let $\Phi:=(\pi,\mathbf{f}):\cC\lr \cK_4\times \PP^2$ be the push-out map.
Then the direct image sheaf $\cF:={\Phi}_*\cO_{\cC}$ is the flat family of pure sheaves parameterized by $\cK_4$. The general sheaves parameterized by $\cK_4$ is stable by Lemma \ref{stable}. Also, $\rH^0(\cF|_{[f]\times \PP^2})\cong\rH^0(f_*\cO_C)= \CC$.  The claim is proved by the below Lemma \ref{flat}.
\end{proof}
\begin{lemma}\label{flat}
Let $S$ be a small disk in $\CC$ containing $0$.
Let $\cF$ be a flat family of sheaves on $$\PP^r\times S\lr S$$ with Hilbert polynomial $\chi(\cF(m))=dm+1$.
Let $\cF$ be stable sheaves over $S$ except the origin $0\in S$. Furthermore, $h^0(\cF|_s)=1$ for all $s\in S$.
Then the modified \emph{stable} sheaf at the origin has a unique non-zero global section.
\end{lemma}
\begin{proof}
Following the proof of \cite[Theorem 2.B.1]{HuLe10}, it is enough to show that, after the single elementary modification of sheaf $\cF$ along $0\in S$, the resulting center fiber has a one-dimensional global section space.
Let $$G\subset F:=\cF|_0$$ be the non-zero maximal destabilizing subsheaf of $F$. Then we have a short exact sequence
$$
\ses{G}{F}{H:=F/G}.
$$
Let $$\chi(H(m))=d'm+\chi' \mbox{ and } \chi(G(m))=(d-d')m+1-\chi'.$$ Then by the maximality of $G$, we have
$$
\frac{1-\chi'}{d-d'}\geq\frac{1}{d} \Leftrightarrow \frac{d'}{d}\geq \chi'.
$$
Since $h^1(F)=0$ and $G$ is supported on a curve in $\PP^r$, $h^1(H)=0$. That is, $h^0(H)=\chi'\geq 0$. Also, by above inequality, $\chi'\leq1$. If $\chi'=1$, then it is contradict to the assumption of $G$. Thus $h^0(H)=0$. This implies that $h^1(G)=0$ because $h^1(F)=0$.
Now, for the modified sheaf $\cF'$ on $\PP^r\times S$, the central fiber fits into the short exact sequence
$$
\ses{H}{\cF'|_{\PP^r\times \{0\}}}{G}.
$$
Therefore $h^1(\cF'|_{\PP^r\times \{0\}})=0$. That is, $h^0(\cF'|_{\PP^r\times \{0\}})=1$.
\end{proof}

\subsection{Proof of Theorem \ref{thm:mainpropintro}}
One of simple observation is that the canonical form of rational quartic plane curve with the prescribed singular points ($[1:0:0], [0:1:0], [0:0:1]$) can be written as
\[
F(x,y,z)=a_0x^2y^2+a_1y^2z^2+a_2z^2x^2+a_3x^2yz+a_4xy^2z+a_5xyz^2
\]
where $x,y,z$ are the homogenous coordinate of $\PP^2$ and $a_0,a_1, \cdots, a_5\in \CC$. 
Hence the projection map $\pi^+: \bM_4^{+}\lr \bH[3]$ in the diagram \eqref{diagram} is generically $\PP^5$-fiberation over $\bH[3]$.
\begin{lemma}\label{keylemm}
The map $$\pi^{+}: \bM_4^{+}\lr \bH[3]$$has a $\PP^5$-fibration over $\bH[3]\setminus D_1\sqcup D_3$.
\end{lemma}
\begin{proof}
Let $\cZ$ be the universal subscheme of $\bH[3]$ and $q:\bH[3]\times \PP^2\lr \bH[3]$ be the projection onto the first component. Then,
the push-forward sheaf $\cE:=q_* (I_{\cZ}^2\boxtimes \cO_{\PP^2}(4))$ is locally free of rank $6$ away from the union $D_1\sqcup D_3$. Indeed, it is straightforward to check that 
\[ h^0(I_Z^2(4)) =\left\{ \begin{array}{ll}
         6 & \mbox{if $[Z]\in \bH[3]\setminus D_1\sqcup D_3$};\\
        5 & \mbox{if $[Z]\in D_1$};\\
        7 & \mbox{if $[Z]\in D_3$}.\end{array} \right. \]
and thus the claim is proved. Now let $\PP(\cE):=\proj(\mathrm{Sym}^{\bullet} (\cE))\lr \bH[3]$ be the projective scheme over $\bH[3]$. Then $\PP(\cE)$ is a subscheme of $\bM_4^{+}$ (cf. \cite{He98}). Since $R_4$ is irreducible and the canonical form of rational quartic curve, $R_4$ is a subscheme of $\PP(\cE)$. Thus $\bM_4^{+}=\overline{R_4}\subset \PP(\cE)$. Therefore $\bM_4^{+}$ is also a $\PP^5$-fiberation away from $D_1\sqcup D_3$.
\end{proof}

\begin{proposition}\label{mainprop}
\begin{enumerate}
\item The map \[\pi^{-}:\bM_4\lr \bN\] is a $\PP^5$-fiberation over $\bN\setminus t(D_3)$.
\item Furthermore, the birational morphism $p_R: \bM_4^{+}\lr \bM$ is a smooth blow-down of $(\pi^{+})^{-1}(D_1)$ such that there exists a commutative diagram
\[
\xymatrix{\bM_4^{+}\ar[d]_{\pi^{+}}\ar[r]^{p_R}&\bM_4\ar[d]^{\pi^{-}}\\
\bH[3]\ar[r]^{t}&\bN.}
\]
\end{enumerate}
\end{proposition}
\begin{proof}
Recall that the image of the blow-up center $Y^{-}$ by the forgetful map $r:\rM^{-}(4m+1)\lr \rM(4m+1)$ is isomorphic to a $\PP^2$-bundle over $\PP^2\times \PP^9$ where the base space parameterizes the pair $(L,C)$ and the fiber is $\PP(\Ext^1(\cO_L,\cO_C))$ for the line $L$ and cubic curve $C$.
Considering the image of the stable maps parameterized by $\cK_4$, the support of the sheaf $F\in \bM\cap r(Y^{-})$ must be $\text{Supp}(F)=L^2\cup Q$ for some line $L$ and conic $Q$.

Step 1. Through the elementary modification of stable maps, we prove that
\[
\bM\cap r(Y^{-})\cong \PP^2\times \PP^5:=\Gamma
\]
where $\PP^2$ (resp. $\PP^5$) parameterizes lines (resp. coincs) in $\PP^2$. Main stream of the proof is the same as that of \cite[Proposition 5.10]{CCM14}. For the convenience of reader we address the detail.
Let $\Delta$ be the locus of the stable maps $f:C\lr \PP^2$ of degree $4$ such that the image is $f(C)=L\cup Q$ where $L$ (resp. $Q$) is a line (resp. conic). One can regard
$f$ as the composition map $$f=g\circ h:C \xrightarrow{h} C' \xrightarrow{g} \PP^2,$$ where $C'$ is a pair of lines.
Here $h$ is a $2:1$-covering from $L_1$ and a bijection from $L_2$.
Also, $g$ is a partial normalization map such that the image is $g(C')=L\cup Q$.
By applying the octahedron axiom for the complex $f^*\Omega_{\PP^2}\lr h^*\Omega_{C'}\lr \Omega_{C}$,
we see that the normal space of $\Delta$ in $\bM(\PP^2,4)$ at $[f]$ is identified with
\[
N_{\Delta/\bM(\PP^2,4),[f]} \cong \Ext_{C'}^1([g^*\Omega_{\PP^2} \to \Omega_{C'}],O_L(-1)).
\]
On the other hand, let $i:L\subset C'$ is the inclusion map. By the octahedron axiom for the complex $(i\circ g)^*\Omega_{\PP^2}\lr i^*\Omega_{C'}\lr \Omega_L$ again, one can easily see that
there exists a quasi-isomorphism between
\[ i^*[g^*\Omega_{\PP^2} \to \Omega_{C'}]\cong N^*_{L/\PP^2}(-1)[1]\]
and thus 
\[
N_{\Delta/\bM(\PP^2,4),[f]}\cong \Hom(N^*_{L/\PP^2}(-1), \cO_L(-1))\cong \Hom (I_L(-1),\cO_L(-1)).
\]
By analyzing the Kodaira-Spencer map 
\begin{equation}\label{restrictdksmap}
\phi:N_{\Delta/\bM(\PP^2,4),[f]}\to \Ext^1(g_*\cO_C,\cO_L(-1)),
\end{equation}
one can prove that $\phi$ is injective and thus each modified sheaf is non-split (cf. \cite[Lemma 4.10]{CK11}).
Hence after the modification, we have a flat family of stable sheaves parameterized by $\PP(N_{\Delta/\bM(\PP^2,4)})=\PP^1$-bundle over $\Delta$. For $w\in \PP^1$, let $F_w$ be the stable sheaf which fits into a exact sequence
\[
\ses{\cO_L(-1)}{F_w}{g_*\cO_C}.
\]
By Proposition \ref{sectionspace}, $h^0(F_w)=1$ for each point $[w]\in \PP^1$. Hence $F_w\cong F_{w'}$
for any point $[w],[w'] \in \PP^1$ by B\'ezout's Theorem (Remark \ref{remsection}).
By some diagram chasing, we know that $F_{w}$ fits into the short exact sequence
\[
\ses{\cO_{L\cup Q}}{F_w}{\cO_{L}}
\]
and thus the claim is proved. Hence the composition map $t\circ \pi^{+}\circ p^{+}: \widetilde{\rM(4m+1)}\lr \bN$ factors through $\bM_4^{-}=\bM_4$ by the
rigidity lemma. That is, we have a morphism
\[
\pi^{-}: \bM_4\lr \bN.
\]
Note that the fiber of $[L]\in t(D_1)\subset \bN$ is constantly a $\PP^5$ which parameterizes conics. By Lemma \ref{keylemm}, we know that the map $\pi^{-}$ is a $\PP^5$-fiberation except the locus $D_3$.

Step 2. By (1), the normal space of $\Gamma$ in $\bM_4$ at $[F_w]\in \bM_4$ is naturally isomorphic to
\[
N_{\Gamma/\bM_4, F_w} \cong N_{t(D_1)/\bN,[\pi^{-}(F_w)]}.
\]
This implies that the blow-up of $\bM_4=\bM_4^{+}$ along $\Gamma$ is nothing but the strict transformation of the blow-up morphism $p^{-}:\widetilde{\rM(4m+1)}\lr \rM^{-}(4m+1)$ because the projectivation of the normal space $N_{t(D_1)/\bN,[\pi^{-}(F_w)]}$ determines three points on the line $L$. In special, the exceptional divisor $E$ in $\widetilde{\bM}^{+}$ is a $\PP^3$-bundle over $\Gamma$. Then the image $p^{+}(E)\cong E\subset Y^{+}$ is still divisor in $\bM^{+}$ and thus $\widetilde{\bM}^{+}\stackrel{p^{+}}\cong \bM^{+}$. Note that $E$ is the inverse image $(\pi^{+})^{-1}(D_1)$ by its construction. This finishes the proof of (2).
\end{proof}
\section{Poincar\'e Polynomial of $\bM_4$}
In this section we compute the Poincar\'e polynomial of $\bM_4$ by using Proposition \ref{mainprop}.
\subsection{Fibre over the non-curvilinear points}
In this section, we describe the fiber of the morphism $p$ over $Y_1$ by using the computer algebra system, \texttt{Macaulay2} (\cite{GS}) and hence we compute the virtual Poincar\'e polynomial of $\bM_4$.
\begin{proposition}\label{specialfiber}
Each fiber $\pi^{-}:\bM_4\lr \bN$ over the closed point $[Z]\in D_1$ is isomorphic to a projective space $\PP^8$.
\end{proposition}
\begin{proof}
We assume that $I_Z=\langle x^2,xy,y^2\rangle$.
We prove the claim by presenting a local neighborhood of $\bM_4^{+}$ containing the fiber $p^{-1}([Z])$ with the help of the computer system.
Let $a_1,a_2,b_1,b_2,c_1,c_2$ be a local affine chart of $\bH[3]$ at $[I_Z]=[\langle x^2,xy,y^2\rangle]$ (\cite[Section 3.1]{Hai02}). The ideal of subschemes parameterizing the points of $\bH[3]$ around $[Z]$ is generated by three quadric polynomials:

\begin{enumerate}[i)]
\item $q_1:=x^2-a_1x-a_2y-a_2(b_1-c_2)-b_2(b_2-a_1)$;
\item $ q_2:=xy-b_1x-b_2y-a_2c_1+b_1b_2$;
\item $ q_3:=y^2-c_1x-c_2y-c_1(b_2-a_1)-b_1(b_1-c_2)$.
\end{enumerate}
Note that $\langle q_1,q_2,q_3\rangle|_{\{a_i=b_i=c_i=0\}}=I_Z$ for $i=1$, $2$.
Let 
\[f(x,y,z)=d_0x^4+d_{1}x^3y+d_{2}x^2y^2+d_3xy^3+\cdots +d_{8}y^3z+d_{9}x^2z^2 +d_{10}y^2z^2+d_{11}xyz^2+d_{12}xz^3+d_{13}yz^3+d_{14}z^{4}\]
 be the homogeneous polynomial defining the plane quartic curve. Let $f_x$ be the partial derivates of $f$ with respect to $x$ (similarly, $f_y$ and $f_z$). The rationality condition of the plane curve provides us the nine resultant conditions 
\[J=\langle \mathrm{Res}(f_x,q_i), \mathrm{Res}(f_y,q_i), \mathrm{Res}(f_z,q_i)\rangle\qquad \mathrm{for} \qquad i=1,2,3.\]
Note that $f_x(x,y,1)\equiv s(x,y) \mod(q_1,q_2, q_3)$ for some quadric polynomial $s(x,y)$ (similarly, $f_y(x,y,1)$ and $ f_z(x,y,1)$). Hence the resultant can be easily computed by using the computer algebra system, \texttt{Macaulay2} (\cite{GS}) and obtain that the ideal $J$ is the defining equation of $\bM^{+}$ around $(\pi^{+})^{-1}([Z])$. The special fiber $(\pi^{+})^{-1}([Z])$ is nothing but \[J+\langle a_1, a_2, b_1,b_2, c_1,c_2\rangle=\langle d_{9}, d_{10},d_{11},d_{12}, d_{13},d_{14} \rangle\]
which defines the space $\PP^8$. Thus we finish the proof of the claim.
\end{proof}
\begin{remark}
The fiber $[Z]$ by $\pi^{-}$ in the statement of Proposition \ref{specialfiber} can be described by the stable map. As we have seen in Example \ref{conexam}, the conductor ideal of the map $$f:\PP^1\lr \PP^2, \; f([s:t])=[g_3(s,t)s:g_3(s,t)t:g_4(s,t)]$$ is given by $\langle x^2,xy,y^2\rangle$ for
the general choice of $g_3(s,t)$ and $ g_4(s,t)$. Also the supporting curve $C=f(\PP^1)$ is defined by the ideal $I_C=\langle g_3(x,y)z-g_4(x,y) \rangle$.
Hence the fiber $p^{-1}([Z])$ contains an open part of $\PP(\rH^0(\cO_{\PP^1}(3) \oplus \cO_{\PP^1}(4)))$.
\end{remark}

\subsection{The virtual Poincar\'e ploynomial of $\bM_4$}\label{motivic}
In this subsection, we caculate the virtual Poincar\'e polynomial of $\bM$ by using Theorem \ref{thm:mainpropintro}. Let $X$ be a quasi-projective variety. For the Hodge-Deligne polynomial $\rE_c(X)(u,v)$ for compactly supported cohomology of $X$, let
\[ \rP(X):=\rE_c(X)(-t,-t)\]
be the \emph{virtual} Poincar\'e polynomial of $X$. For the motivic properties of the virtual Poincar\'e polynomial, see \cite[Theorem 2.2]{Mun08} and \cite[Lemma 3.1]{BJ12}.

\begin{proposition}\label{motivicinvariant}
\begin{enumerate}
\item $\rP(\PP^n)=\frac{t^{2n+2}-1}{t^2-1}.$
\item $\rP(X)=\rP(Z)+\rP(X\setminus Z)$ for any closed subset $Z\subset X$. \label{addition}
\item $\rP(X)=\rP(F)\cdot \rP(B)$ for the \'etal locally trivial fiberation $X\lr B$ with the constant fiber isomorphic to the Grassmannian variety $\mathrm{Gr}(k,n)$. Here $\mathrm{Gr}(k,n)$ parameterizes $k$-dimensional sub-vector spaces in the vector space $\CC^n$.
\end{enumerate}
\end{proposition}

\begin{corollary}\label{maincor}
The virtual Poincar\'e polynomials of $\bM_4$ is given by 
\[
1+2t^2+5t^4+8t^6+11t^8+12t^{10}+13t^{12}+13t^{14}+11t^{16}+7t^{18}+3t^{20}+t^{22}.
\]
\end{corollary}
\begin{proof}
By $P(\bH[3])=1 + 2t^2 + 5t^4 + 6t^6 + 5t^8 + 2t^{10} + t^
{12}$ (\cite{ELB88}) and Proposition \ref{prop:hilbcon},
\[P(\bN)=P(\bH[3])+(1-P(\PP^3))\cdot P(\PP^2)=1+t^2+3t^4+3t^6+3t^8+t^{10}+t^{12}.\]
Also, from Proposition \ref{mainprop}, Proposition \ref{specialfiber} and Proposition \ref{motivicinvariant}, we have
\[
P(\bM)=P(\PP^5)\cdot(P(\bN)-P(\PP^2))+P(\PP^8)\cdot P(\PP^2).
\]
Cooking up these ones, we obtain the results.
\end{proof}



\bibliographystyle{alpha}

\end{document}